\documentclass{amsart}

\usepackage{txfonts}
\usepackage{amsmath,amssymb,amsfonts,enumerate,amsthm, amscd}

\renewcommand{\Im}{\mbox{Im}\,}

\newtheorem{theorem}{Theorem}[section]
\newtheorem{lemma}[theorem]{Lemma}
\newtheorem{proposition}[theorem]{Proposition}

\theoremstyle{definition}
\newtheorem{definition}[theorem]{Definition}
\newtheorem{definitions}[theorem]{Definitions}
\theoremstyle{remark}
\newtheorem{remark}[theorem]{Remark}
\newtheorem{example}[theorem]{Example}

\theoremstyle{Definition and Notation}

\begin{document}
\bibliographystyle{amsplain}

\title[When every Gorenstein projective (resp. flat) module is strongly...]{When every Gorenstein projective (resp. flat) module is
strongly Gorenstein projective (resp. flat).}

\author{Najib Mahdou}
\address{Najib Mahdou\\Department of Mathematics, Faculty of Science and Technology of Fez, Box 2202, University S.M. Ben Abdellah Fez, Morocco.}

\author{Mohammed Tamekkante}
\address{Mohammed Tamekkante\\Department of Mathematics, Faculty of Science and Technology of Fez, Box 2202, University S.M. Ben Abdellah Fez, Morocco.}



\begin{abstract}

In \cite{Ouarghi}, the authors discuss the rings over which all
modules are strongly Gorenstein projective. In this paper, we are
interesting to an extension of this idea. Thus, we discuss the
rings over which every Gorenstein projective (resp. flat) module
is strongly Gorenstein projective (resp, flat). Our aim is to give
examples of rings with different Gorenstein global dimension
satisfied this condition.
\end{abstract}

\maketitle

\section{Introduction}
Throughout this paper, all rings are commutative with identity
element, and all modules are unital.\\

$\mathbf{\quad Setup\; and\; Notation:}$ Let $R$ be a ring, and
let $M$ be an $R$-module. As usual we use $pd_R(M)$, $id_R(M)$ and
$fd_R(M)$ to denote, respectively, the classical projective,
injective and flat dimensions of $M$. By $gldim(R)$ and $wdim(R)$
we denote, respectively, the classical global and weak dimensions
of R.\\
It is by now a well-established fact that even if R to be
non-Noetherian, there exists Gorenstein projective, injective and
flat dimensions of M, which are usually denoted by $Gpd_R(M)$,
 $Gid_R(M)$, and $Gfd_R(M)$, respectively. Some references are
 \cite{Bennis and Mahdou1, Bennis and Mahdou2, Christensen, Christensen
 Frankild, Enochs, Enochs2, Eno Jenda Torrecillas, Holm}.\\

 Recently in \cite{Bennis and Mahdou2}, the authors started the study of
global Gorenstein dimensions of rings, which are called, for a
commutative ring $R$, Gorenstein global  projective, injective,
and weak dimensions of $R$, denoted by $GPD(R)$, $GID(R)$, and
$G.wdim(R)$, respectively, and, respectively, defined as
follows:\bigskip

$\begin{array}{cccc}
  1) & GPD(R) & = & sup\{ Gpd_R(M)\mid M$ $R-module\} \\
  2) & GID(R) & = & sup\{ Gid_R(M)\mid M$ $R-module\} \\
  3) & G.wdim(R) & = & sup\{ Gfd_R(M)\mid M$ $R-module\}
\end{array}$
\\

They proved that, for any ring R, $ G.wdim(R)\leq GID(R) = GPD(R)$
(\cite[Theorems 2.1 and 2.11]{Bennis and Mahdou2}). So, according
to the terminology of the classical theory of homological
dimensions of rings, the common value of $GPD(R)$ and $GID(R)$ is
called
Gorenstein global dimension of $R$, and denoted by $G.gldim(R)$.\\
They also proved that the Gorenstein global and weak dimensions
are refinement of the classical global  and weak dimensions of
rings. That is : $G.gldim(R) \leq gldim(R)$ and $G.wdim(R)\leq
wdim(R)$ with equality if $wdim(R)$ is finite (\cite[Propositions
2.12]{Bennis and Mahdou2}).\bigskip

In \cite{Bennis and Mahdou1}, the authors studied a particulars
cases of a Gorenstein projective, injective and flat modules which
they call a strongly Gorenstein projective, injective and flat
modules respectively, and defined as follows:
\begin{definitions}\
\begin{enumerate}
    \item A module $M$ is said to be strongly Gorenstein projective ($SG$-projective for short), if
there exists an exact sequence of   projective modules  of the
form:
$$\mathbf{P}=\ \cdots\rightarrow P\stackrel{f}\rightarrow
P\stackrel{f} \rightarrow P\stackrel{f}
         \rightarrow P \rightarrow\cdots$$  such that  $M \cong \Im(f)$
and such that $Hom(-,P)$ leaves the sequence $\mathbf{P}$ exact whenever $P$ is projective.\\
  The exact sequence $\mathbf{P}$ is
called a strongly complete projective resolution.
    \item The strongly Gorenstein injective modules
are defined dually.
    \item A module $M$ is said to be strongly
Gorenstein flat ($SG$-flat for short), if there exists an exact
sequence of flat module  of the form:
$$\mathbf{F}=\ \cdots\rightarrow F\stackrel{f}\rightarrow F \stackrel{f}\rightarrow
F \stackrel{f}\rightarrow F \rightarrow\cdots$$ such that  $M
\cong \Im(f)$ and such that $-\otimes I$ leaves $\mathbf{F}$ exact
whenever $I$ is injective.

The exact sequence $\mathbf{F}$ is called a strongly complete flat
resolution.
\end{enumerate}
\end{definitions}
The principal role of the strongly Gorenstein projective and
injective modules is to give a simple characterization of a
Gorenstein projective and injective modules, respectively, as
follows:
\begin{theorem}[\cite{Bennis and Mahdou1}, Theorem 2.7] A  module is Gorenstein projective (resp., injective)
if, and only if, it is a direct summand of a strongly Gorenstein
projective (resp., injective) module.

\end{theorem}
Using \cite[Theorem 3.5]{Bennis and Mahdou1} together with
\cite[Theorem 3.7]{Holm}, we have the next result:
\begin{proposition} Let $R$ be a coherent ring. A module is Gorenstein flat
if, and only if, it is a direct summand of a strongly Gorenstein
flat module.
\end{proposition}

$\mathbf{Notation.}$ By  $\mathcal{P}(R)$, $\mathcal{I}(R)$ and
$\mathcal{F}(R)$ we denote the classes of all projective,
injective and fat $R$-modules respectively and by
 $\mathcal{GP}(R)$,
$\mathcal{GI}(R)$ and $\mathcal{GF}(R)$ denote the classes of all
strongly Gorenstein projective, injective and fat $R$-modules
respectively. Furthermore, we let $\mathcal{SGP}(R)$,
$\mathcal{SGI}(R)$ and  $\mathcal{SGF}(R)$
denote the classes of all  Gorenstein  projective, injective  and flat $R$-modules, respectively.\\

In  many place of this papers we use the notion of resolving
class. This notion was be introduced by Holm in \cite{Holm} as
follows:
\begin{definition}[Definition 1.1,\cite{Holm}] For any class $\mathcal{X}$ of $R$-modules.
\begin{description}
    \item[a/]We call $\mathcal{X}$ projectively resolving if $\mathcal{P}(R)\subseteq \mathcal{X}$, and for every short exact sequence
$0 \longrightarrow X' \longrightarrow X \longrightarrow X"
\longrightarrow 0$ with $X" \in \mathcal{X}$ the conditions $X'
\in \mathcal{X}$ and $X \in \mathcal{X}$ are equivalent.
    \item[b/]We call $\mathcal{X}$ injectively resolving if $\mathcal{I}(R)\subseteq \mathcal{X}$, and for every short exact sequence
$0 \longrightarrow X' \longrightarrow X \longrightarrow X"
\longrightarrow 0$ with $X' \in \mathcal{X}$ the conditions $X"
\in \mathcal{X}$ and $X \in \mathcal{X}$ are equivalent.
\end{description}
\end{definition}
In \cite{Holm} again, Holm prove that  the class $\mathcal{GP(R)}$
is projectively resolving and closed under arbitrary direct sums
and under direct summands (\cite[Theorem 2.5]{Holm}), and dually,
the class $\mathcal{GI(R)}$ is injectively  resolving and closed
under arbitrary direct products and under direct summands
(\cite[Theorem 2.6]{Holm}). He also prove that, if $R$ is
coherent, then the class $\mathcal{GF(R)}$ is projectively
resolving and closed under direct summands (\cite[Theorem
3.7]{Holm}).\\

For any ring $R$, it is clear that we have the following
inclusions of classes:

$$\mathcal{P}(R)\subseteq \mathcal{SGP}(R)\subseteq \mathcal{GP}(R)$$

We will see after that the inverse inclusions are not true in
general. Indeed, using \cite[Theorem 2.7]{Bennis and Mahdou1}, the
inclusion $\mathcal{SGP}(R)\subseteq \mathcal{P}(R)$ means that
$G.gldim(R)=gldim(R)$ and so we are in the classical case. \\
In this paper, we are interesting to discuss the equality of
classes $\mathcal{SGP}(R)= \mathcal{GP}(R)$ i.e, when every
Gorenstein projective module is strongly Gorenstein projective. A
trivial example is when the global dimension of $R$ is finite. In
the second part of the second section, we are also interesting to
the equality $\mathcal{SGF}(R)= \mathcal{GF}(R)$ over a coherent
ring $R$. In many place of this papers, we use the following
Lemma.
\begin{lemma}\label{lemma1}
Consider the following diagram of modules aver a ring $R$.
$$(\star)\begin{array}{ccccccc}
  0\rightarrow & M & \stackrel{\alpha}\rightarrow & P & \stackrel{\beta}\rightarrow & M & \rightarrow 0 \\
   & _{u}\downarrow &  &  &  & _{u}\downarrow &  \\
  0\rightarrow & Q & \stackrel{\iota}\rightarrow & Q\oplus Q & \stackrel{j}\rightarrow & Q & \rightarrow 0 \\
\end{array}$$
where $M$ is a Gorenstein projective module, $P$ and $Q$ are
projective and $\iota$ and $j$ are the canonical injection and
projection respectively. Then, there is a morphism
$\gamma:P\rightarrow Q\oplus Q$ which complete $(\star)$ and make
it commutative.
\end{lemma}
\begin{proof}
If we apply the functor $Hom(-,Q)$ to the short exact sequence
$$(\ast)\quad 0\rightarrow  M  \stackrel{\alpha}\rightarrow  P
\stackrel{\beta}\rightarrow  M  \rightarrow 0$$ we obtain the
short exact sequence:
$$(\ast\ast)\quad 0\rightarrow  Hom(M,Q)  \stackrel{\circ\beta}\rightarrow
Hom(P,Q) \stackrel{\circ\alpha}\rightarrow  Hom(M,Q)  \rightarrow
0$$ Since $Ext(M,Q)=0$ (\cite[Proposition 2.3]{Holm}). On the
other hand, $u\in Hom(M,Q)$. Then, from the exactness of
$(\ast\ast)$, there is a morphism $\upsilon: P\rightarrow Q$ such
that $\upsilon\circ \alpha=u$. Consequently, we can verified that
the morphism $\gamma:P\rightarrow Q\oplus Q$ defined by
$\gamma(p):=(\upsilon(p),u\circ\beta(p))$ whenever $p\in P$ is the
desired morphism.
\end{proof}
Dually, one can find easily the injective version of Lemma
\ref{lemma1}.\\ The aim of this paper is to construct a family of
rings $\{R_i\}_{i}$ over which every Gorenstein projective module
is strongly Gorenstein projective and such that $G.gldim(R_i)=i$
and $wdim(R_i)=\infty$.
\section{Main results}

Now, we give our main first result. \\

\begin{theorem}\label{theorem1}
Let $R$ be a ring. The following are equivalents:
\begin{enumerate}
    \item Every Gorenstein projective
module is strongly Gorenstein projective.
    \item For every module $M$ with $Gpd(M)\leq 1$, there
    exists a short exact sequence $$0\rightarrow M \rightarrow Q
    \rightarrow M \rightarrow 0$$ where $pd(Q)\leq 1$.
\end{enumerate}
\end{theorem}
\begin{proof}
Assume $(1)$ and we claim $(2)$. Let $M$ be a module such that
$Gpd(M)\leq 1$ and pick a short exact sequence $0\rightarrow G
\rightarrow Q \rightarrow M \rightarrow 0$ where $Q$ is projective
and $G$ is Gorenstein projective module (then strongly Gorenstein
projective by the hypothesis condition). Thus, from
\cite[Proposition 2.9]{Bennis and Mahdou1}, there exists a short
exact sequence $0\rightarrow G\rightarrow P\rightarrow G
\rightarrow 0$ where $P$ is projective. So, using Lemma
\ref{lemma1}, there is a morphism $\gamma: P\rightarrow Q\oplus Q$
such the following diagram is commutative:

$$(\star)\begin{array}{ccccccc}
  0\rightarrow & G & \rightarrow & P & \rightarrow & G & \rightarrow 0 \\
   & \downarrow &  & _{\gamma}\downarrow &  & \downarrow &  \\
  0\rightarrow & Q & \stackrel{\iota}\rightarrow & Q\oplus Q & \stackrel{j}\rightarrow & Q & \rightarrow 0 \\
\end{array}$$
 Hence, applying the Snake Lemma to this diagram and since $M\cong coker(G\rightarrow Q)$, we obtain a
 short  exact sequence of the form $0\rightarrow M \rightarrow X
 \rightarrow M \rightarrow 0$ where $X\cong coker (\gamma)$.
 Clearly, we have $pd(X)\leq 1$, as desired.\\
 Conversely, we assume $(2)$ and we claim $(1)$. Let $M$ be a
 Gorenstein projective module. By the hypothesis condition, there
 is an exact sequence $(\star) 0\rightarrow M \rightarrow Q \rightarrow M
 \rightarrow 0$ where $pd(Q)\leq 1$. But the class $\mathcal{GP}(R)$ is
 projectively resolving (by \cite[Theorem 2.5]{Holm}). Hence, from $(\star)$, $Q$ is also
 Gorenstein projective. Consequently, from \cite[Proposition
 2.27]{Holm}, $Q$ is projective. On the other hand, since $M$ is
 Gorenstein projective, for any projective module $P$ and any
 integer $i>0$, $Ext^i(M,P)=0$. Thus, from \cite[Proposition 2.9]{Bennis and
 Mahdou1}, $M$ is strongly Gorenstein projective, as desired.
 \end{proof}

In the next, we give example of rings over which
$\mathcal{SGP}(-)=\mathcal{GP}(-)$ or different.

\begin{example}\label{basic example} Let $D$ be a principal ideal domain and $\mathcal{P}$ a nonzero
prime ideal of $D$. Consider the rings $R=D/\mathcal{P}^2$
and $S=D/\mathcal{P}^3$. Then,
\begin{enumerate}
    \item $\mathcal{SGP}(R)=\mathcal{GP}(R)$, and
    \item $\mathcal{SGP}(S)\varsubsetneq \mathcal{GP}(S)$.
\end{enumerate}
\end{example}
\begin{proof}
\begin{enumerate}
    \item Follows immediately from \cite[Corollary 3.9]{Ouarghi}
    \item From \cite[Corollary 3.10]{Ouarghi}, $\mathcal{SGP}(S)\neq
    \mathcal{M}(S)=\mathcal{GP}(S)$ (where $\mathcal{M}(S)$ is the class of
    all $S$-modules). So, we have the desired strict inclusion.
\end{enumerate}
\end{proof}

From \cite[Corollary 3.9 and 3.10]{Ouarghi}, the rings $R$ and $S$
of Example \ref{basic example} have zero Gorenstein Global
dimension. In order to move up this $G.gldim$, we past by the
direct product of rings as shown by the next result.

\begin{theorem}\label{direct product of rings}
Let $\{R_i\}_{i=1,...,m}$ be a family of rings with finite
Gorenstein global dimensions. Then, $\mathcal{SGP}(\displaystyle
\prod_{i=1}^mR_i)=\mathcal{GP}(\displaystyle \prod_{i=1}^mR_i)$
if, and only if, $\mathcal{SGP}(R_i)=\mathcal{GP}(R_i)$ for each
$i=1,...,m$.
\end{theorem}
\begin{proof}
By induction on $m$, we may assume $m = 2$.\\
Assume that  $\mathcal{SGP}(R_1\times R_2)=\mathcal{GP}(R_1\times
R_2)$ and let $M$ be a $G$-projective $R_1$- module. We claim that
$M$ is a strongly Gorenstein projective $R_1$-module. Obviously,
$M\times 0$ is an $R_1\times R_2$-module (see \cite[Page
101]{Berrick}). First, we claim that  $M\times 0$ is a
$G$-projective  $R_1\times R_2$-module. The $R_1$-module $M$ is a
direct summand of an $SG$-projective $R_1$-module $N$
(\cite[Theorem 2.7]{Bennis and Mahdou1}). For such module, there
is a short exact sequence of $R_1$-modules $0\rightarrow
N\rightarrow P \rightarrow N \rightarrow 0$  where $P$ is a
projective $R_1$-module (\cite[Proposition 2.9]{Bennis and
Mahdou1}). Hence, we have a short sequence of $R_1\times
R_2$-modules $(\ast)\quad 0\rightarrow N\times 0\rightarrow
P\times 0 \rightarrow N\times 0 \rightarrow 0$ and $P\times 0$ is
a projective $R_1\times R_2$-module. But $G.gldim(R_1\times R_2)$
is finite (by \cite[Theorem 3.1]{Bennis and Mahdou3}). Then, there
is an integer $i>0$ such that $Ext_{R_1\times R_2}^i( N\times
0,Q)=0$ for each projective $R_1\times R_2$-module $Q$ (since
$Gpd_{R_1\times R_2}(N\times 0)<\infty$ and by \cite[Theorem
2.20]{Holm}). From $(\ast)$ we deduce that $Ext_{R_1\times R_2}(
N\times 0,Q)=0$. Then, by \cite[Proposition 2.9]{Bennis and
Mahdou1}, $N\times0$ is an $SG$-projective $R_1\times R_2$-module.
So, $M\times0$ is a $G$-projective $R_1\times R_2$-module (since
it is a direct summand of $N\times0$ as an $R_1\times R_2$-modules
and by \cite[Theorem 2.5]{Holm}). Now, we claim that $M$ is an
$SG$-projective $R_1$-module. By hypothesis, the $R_1\times
R_2$-module $M\times 0$ is $SG$-projective (since $M\in
\mathcal{GP}(R_1\times R_2)=\mathcal{SGP}(R_1\times R_2)$). Thus,
there exists
    a short exact sequence of $R_1\times R_2$-modules $(\star)\quad 0
    \rightarrow M\times 0\rightarrow P \rightarrow
    M \times 0 \rightarrow 0$ where $P$ is a  projective $R_1\times
    R_2$-module. Now, tensor $(\star)$ by $-\otimes R_1$ we obtain the short exact sequence of $R_1$-modules (see that $R_1$ is a projective
$R_1\times R_2$-module) $$(\star\star)\quad 0
    \rightarrow M\times 0 \displaystyle\otimes _{R_1\times R_2}R_1 \rightarrow P\displaystyle\otimes _{R_1\times R_2}R_1 \rightarrow
    M\times 0 \displaystyle\otimes _{R_1\times R_2}R_1 \rightarrow
    0$$ But  $M\times 0 \displaystyle\otimes _{R_1\times
    R_2}R_1\cong M\times0\displaystyle\otimes_{R_1\times
    R_2}(R_1\times R_2)/(0\times R_2) \cong M\times 0\displaystyle\cong_RM$ (isomorphism of $R$-modules). Then, we can
    write $(\star\star)$ as $ 0 \rightarrow M
    \rightarrow P\displaystyle\otimes _{R_1\times
    R_2}R_1 \rightarrow M \rightarrow 0$. It is clear that $P\displaystyle\otimes
    _{R_1\times
    R_2}R_1$ is a projective $R_1$-module. Furthermore, by \cite[Theorem 2.20]{Holm}, $Ext_{R_1}(M,F)=0$ for every $R_1$-module projective
      $F$ since $M$ is a $G$-projective $R_1$-module. So, by \cite[Proposition 2.9]{Bennis and Mahdou1}, $M$ is an
    $SG$-projective $R_1$-module, as desired.\\
   Similarly, we can prove that $\mathcal{SGP}(R_2)=\mathcal{GP}(R_2)$.\\

 Conversely, assume that $\mathcal{SGP}(R_i)=\mathcal{GP}(R_i)$
for $i=1,2$  and let
    $M$ be a $G$-projective $R_1\times R_2$-module. We claim that $M$ is an $SG$-projective $R_1\times R_2$-module. We have the isomorphism of $R_1\times R_2$-modules: $$M\cong M\otimes_{R_1\times
    R_2}R_1\times R_2\cong M\otimes_{R_1\times
    R_2}(R_1\times0\oplus 0\times R_2)\cong M_1\times M_2$$ where
    $M_i=M\otimes_{R_1\times
    R_2} R_i$ for $i=1,2$ (for more details see \cite[p.102]{Berrick}). By \cite[Lemma 3.2]{Bennis and Mahdou3}, for each $i=1,2$,  $M_i$ is
    a $G$-projective $R_i$-module. Then, by hypothesis, $M_i$ is an $SG$-projective $R_i$-module  for  $i=1,2$.
     On the other hand, the family $\{R_i\}_{i=1,2}$ of rings satisfies the conditions of
    \cite[Lemma 3.3]{Bennis and Mahdou3} (by \cite[Corollary 2.10]{Bennis and Mahdou2}  since $G.gldim(R_i)$ is finite  for each $i=1,2$). Thus,   $M=M_1\times M_2$ is
   an  $SG$-projective  $R_1\times R_2$-module, as desired.
   \end{proof}
   Now, we are able to construct a non-Noetherian family of rings
   $\{R_i\}$ over which every Gorenstein projective module is
   strongly Gorenstein projective and such that $G.gldim(R_i)=i$
   and $wdim(R_i)=\infty$.
\begin{example}\label{example2} Consider a non-semisimple quasi-Frobenius ring
$R=K[X]/(X^2)$ where $K$ is a field, and a non- Noetherian
hereditary ring $S$. Then, for every positive integer $n$, we
have:
\begin{enumerate}
    \item $\mathcal{SGP}(R \times S[X_1,X_2, ...,X_n])=\mathcal{GP}(R \times
S[X_1,X_2, ...,X_n])$,
    \item $G.gldim(R \times S[X_1,X_2, ...,X_n]) = n
+ 1$ and $wdim(R \times S[X_1,X_2, ...,X_n]) = \infty$.
\end{enumerate}
\end{example}
\begin{proof} From \cite[Example 3.4]{Bennis and Mahdou3}, only the
first assertion  need an argument. It is clear that
$\mathcal{P}(S[X_1,X_2,...,X_n])=\mathcal{SGP}(S[X_1,X_2,...,X_n])=\mathcal{GP}(S[X_1,X_2,...,X_n])$
since $wdim(S[X_1,X_2,...,X_n])$ is finite
 (by the Hilbert Syzygies's Theorem) and by using \cite[Proposition 2.27]{Holm}.  On the other hand, from Example \ref{basic example}, $\mathcal{SGP}(R)=\mathcal{GP}(R)$. Thus, Theorem \ref{direct product of
 rings} finish the proof.
\end{proof}
\begin{remark}
Similarly as in Example \ref{example2} and  by using the ring $S$
of Example \ref{basic example} and Theorem \ref{direct product of
rings}, we can construct a family of rings $\{S_i\}_{i\geq 0}$
(Noetherian or not) and with any Gorenstein global dimensions,
such that $\mathcal{SGP}(S_i)\varsubsetneq\mathcal{ GP}(R_i)$.
\end{remark}

 In the rest of this paper we will be interesting to discuss and give examples of rings satisfies the equality of classes $\mathcal{SGF}(-)=\mathcal{GF}(-)$. Before that, the next rest
 result give a particular case where the equality $\mathcal{SGP}(-)=\mathcal{GP}(-)$, $\mathcal{SGI}(-)=\mathcal{GI}(-)$ and $\mathcal{SGF}(-)=\mathcal{GF}(-)$ are equivalents.
 \begin{theorem}\label{class relation}
 Let $R$ be a commutative ring with Gorenstein global dimension $\leq 1$. We
 consider the following assertions:
 \begin{enumerate}
    \item Every Gorenstein projective module is strongly Gorenstein projective.
    \item Every Gorenstein injective module is strongly Gorenstein injective.
 \item Every Gorenstein flat module is strongly Gorenstein flat.
  \end{enumerate}
 Then  (1) and (2) are equivalent. If $R$ is coherent,
 all assertions are equivalent.
 \end{theorem}
 \begin{proof}$1 \Rightarrow 2$. Assume $(1)$ and let $M$ be a  $G$-injective
$R$-module. We claim that $M$ is $SG$-injective. By hypothesis,
$Gpd_R(M)\leq 1$. Then, there exists a short exact sequence of
$R$-modules $(\star) \quad 0\rightarrow M \rightarrow Q
\rightarrow M \rightarrow 0$ where $pd_R(Q)\leq 1$ (by Theorem
\ref{theorem1}). By \cite[Corollary 2.10]{Bennis and Mahdou2},
$id_R(Q)\leq 1$. On the other hand, $\mathcal{GI}(R)$ is
injectively resolving (\cite[Theorem 2.6]{Holm}). Then, $Q$ is a
$G$-injective $R$-module. Thus, by the injective version of
\cite[Proposition 2.27]{Holm}, we conclude that $Q$ is injective.
Moreover, for every injective $R$-module $E$ we have
$Ext_R(E,M)=0$ (since $M$ is $G$-injective). Consequently,  from
\cite[Remark 2.10]{Bennis and Mahdou1},   $M$ is  $SG$-injective, as desired.\\
$2\Rightarrow 1$. By the injective version of Theorem
\ref{theorem1}, the argument of this implication is similar to the
proof of the
first one.\\
Now, we suppose that $R$ is coherent and prove the equivalence $1\Leftrightarrow 3$.\\
$1\Rightarrow 3$. Let $M$ be a $G$-flat module.  By hypothesis $
Gpd_R(M)\leq 1$. Thus, from Theorem \ref{theorem1}, there is an
exact sequence $0\rightarrow M \rightarrow X \rightarrow M
\rightarrow 0$ where $pd(X)\leq 1$. Then,
$id(Hom_\mathbb{Z}(X,\mathbb{Q}/\mathbb{Z}))=fd(X)\leq pd(X)\leq
1$. Furthermore, from \cite[Theorem 3.7]{Holm}, $X$ is Gorenstein
flat module since $\mathcal{GF}(R)$ is projectively resolving.
Hence, by \cite[Proposition 3.11]{Holm},
$Hom_\mathbb{Z}(X,\mathbb{Q}/\mathbb{Z})$ is Gorenstein injective.
Consequently, by the dual of \cite[Proposition 2.27]{Holm},
$Hom_\mathbb{Z}(X,\mathbb{Q}/\mathbb{Z})$ is injective. Then, from
\cite[Theorem 1.2.1]{Glaz}), $X$ is flat. Hence,  $M$ is
immediately $SG$-flat (by \cite[Proposition 3.6]{Bennis and
Mahdou1} and since for any injective module $I$, we have
$Tor(M,I)=0$ since $M$ is
$G$-flat).\\
$3\Rightarrow 1$. Let $I$ be an injective $R$-module. From
\cite[Corollary 2.10]{Bennis and Mahdou2}, $fd_R(I)\leq1$. Then,
from \cite[Theorem 3.8]{Ding}, $R$ is an $1-FC$ ring (i.e.,
coherent ring with $Ext^2_R(P,R)=0$ for each finitely presented
$R$-module $P$). Now,  let $M$ be a $G$-projective $R$-module.
Then, $M$ embeds in projective $R$-module. So, from \cite[Theorem
7]{Chen}, $M$ is $G$-flat. Then, by hypothesis $M$ becomes
$SG$-flat. Hence, there exists a short exact sequence
$0\rightarrow M \longrightarrow F \rightarrow M \rightarrow 0$
where $F$ is flat. By the resolving of
 the class $\mathcal{GP}(R)$ and from the short exact sequence above we
deduce that  $F$ is $G$-projective (since $M$ is $G$-projective).
On the other hand, $pd_R(F)<\infty$ (by \cite[Corollary
2.10]{Bennis and Mahdou2} and since $F$ is flat). Therefore, $F$
is projective by \cite[Proposition 2.27]{Holm}. So $M$ is
$SG$-projective (by \cite[Proposition 2.9]{Bennis and Mahdou1} and
since $Ext(M,P)=0$ for every projective module $P$ because  $M$ is
$G$-projective).
\end{proof}

We have a similar result with the perfect rings as shown by the
next result. Recall that a ring $R$ is called perfect if every
flat $R$-module is projective \cite{Bass}.

\begin{proposition}\label{proposition}
   Let $R$ be a coherent ring with  finite Gorenstein global
   dimension. If every Gorenstein flat module is strongly Gorenstein flat, then every Gorenstein projective module is strongly Gorenstein
   projective with equivalence if  $R$ is perfect.
   \end{proposition}
    \begin{proof} Assume that every Gorenstein flat module is strongly Gorenstein flat and let $M$ be a $G$-projective module. By \cite[Theorem 2.20]{Holm}, $Ext_R(M,Q)=$ for every
   projective module $Q$. So, to prove that $M$ is $SG$-projective it suffices to find a short exact sequence
   of $R$-modules $0  \rightarrow M \rightarrow P \rightarrow M \rightarrow 0$ where $P$ is projective
   (by \cite[Proposition 2.9]{Bennis and Mahdou1}). From
   \cite[Proposition 3.4 and Theorem 3.24]{Holm}, $M$ is also
   $G$-flat (since $G.gldim(R)$ is finite). Then, by the hypothesis condition, $M$ becomes  $SG$-flat. Thus, from \cite[Proposition 3.6]{Bennis and Mahdou1},
   there exists a short exact sequence $0\rightarrow N \rightarrow F \rightarrow M
   \rightarrow 0$ where $F$ is flat. So, from \cite[Corollary 2.10]{Bennis and
   Mahdou2}, $pd_R(F)$ is finite. On the other hand, by \cite[Theorem
   3.7]{Holm}, $\mathcal{GF}(R)$ is projectively resolving and then, from the
   short exact sequence above, $F$ is $G$-projective since $M$ is
   $G$-projective. Therefore, from \cite[Proposition 2.27]{Holm},
   $F$ is projective. Consequently, we have the desired short
   exact sequence.\\
   Now, assume that $R$ is a perfect ring with finite Gorenstein
   global dimension and such that every Gorenstein projective module is strongly Gorenstein projective. Let $M$ a $G$-flat module. We
    claim that it is strongly Gorenstein flat. By \cite[Theorem
   3.14]{Holm}, $Tor_R(M,I)=0$ for every injective module $I$. So,
   from \cite[Proposition 3.6]{Bennis and Mahdou1}, it stays to
   prove the existence of a short exact sequence $0
   \rightarrow M \rightarrow F \rightarrow M
   \rightarrow 0$ where $F$ is flat. By \cite[Theorem 3.5]{Bennis and
   Mahdou1} $M$ is a direct summand of an $SG$-flat module $N$.
   For such module, by \cite[Proposition 3.6]{Bennis and Mahdou1}, there exists a short exact sequence $(\ast) \quad 0
   \rightarrow N \rightarrow F \rightarrow N
   \rightarrow 0$ where $F$ is flat (then projective since $R$
   is perfect). Now, let $P$ be a  projective module.
   We have $id_R(P) \leq n$ where $n=G.gldim(R)$ (by \cite[Corollary 2.10]{Bennis and
   Mahdou2}). Then, $Ext_R^{n+1}(N,P)=0$ and so, from the short exact
   sequence $(\ast)$ we deduce that $Ext_R(N,P)=0$. So, from
   \cite[Proposition 2.9]{Bennis and Mahdou1}, $N$ is
   $SG$-projective. Consequently, $M$ is $G$-projective (since it is
   a direct summand of $N$ and by \cite[Theorem 2.7]{Bennis and
   Mahdou1}). Then, by hypothesis, $M$ becomes  $SG$-projective. Hence, by \cite[Proposition 2.9]{Bennis and Mahdou1}, there
   exists a short exact sequence $0\rightarrow M
   \rightarrow P  \rightarrow M \rightarrow 0$ where
   $P$ is projective (then  flat), and this is the desired short
   exact sequence.
   \end{proof}
   \begin{remark}\label{remark example} Seen   Theorem \ref{class relation}, with the rings $R$ and $S$ of Example  \ref{basic example}, we have
   $S\mathcal{GF}(R)=\mathcal{GF}(R)$ and $\mathcal{SGF}(S)\varsubsetneq \mathcal{GF}(S)$.
   \end{remark}
In what follows, we give the flat version of Theorem  \ref{direct
product of rings} as:
\begin{theorem}\label{direct product of rings2}
Let $\{R_i\}_{i=1,...,m}$ be a family of coherent rings with
finite Gorenstein weak  dimensions. Then,
$\mathcal{SGF}(\displaystyle
\prod_{i=1}^mR_i)=\mathcal{GF}(\displaystyle \prod_{i=1}^mR_i)$
if, and only if, $\mathcal{SGF}(R_i)=\mathcal{GF}(R_i)$ for each
$i=1,...,m$.
\end{theorem}
\begin{proof}
By induction on m, we may assume $m = 2$.\\
First note that $R_1\times R_2$ is coherent since $R_1$ and $R_2$ are coherents.\\

 Assume that every Gorenstein flat $R_1\times R_2$-module is strongly Gorenstein flat and let $M$ be a $G$-flat $R_1$-module. We claim
that $M$ is an $SG$-flat $R_1$-module.  Clearly  $M\times 0$ is an
$R_1\times R_2$-module (see \cite[Page 101]{Berrick}). First, we
claim that $M\times 0$ is a $G$-flat $R_1\times R_2$-module. The
$R_1$-module $M$ is a direct summand of an $SG$-flat $R_1$-module
$N$ (\cite[Theorem 3.5.]{Bennis and Mahdou2}). For such module,
there is a short exact sequence of $R_1$-modules  $0\rightarrow
N\rightarrow F \rightarrow N \rightarrow 0$  where  $F$ is a flat
$R_1$-module. Hence, we have a  short sequence of $R_1\times
R_2$-modules $(\ast)\quad 0\rightarrow N\times 0\rightarrow
F\times 0 \rightarrow N\times 0 \rightarrow 0$ and   $F\times 0$
will be a  flat $R_1\times R_2$-module (\cite[Lemma 3.7]{Bennis
and Mahdou3}). But $G.wdim(R_1\times R_2)$ is finite
(\cite[Theorem 3.5]{Bennis and Mahdou3}). Then, there is an
integer $i>0$ such that $Tor_{R_1\times R_2}^i( N\times 0,I)=0$
for each injective $R_1\times R_2$-module $I$. From $(\ast)$ we
deduce that $Tor_{R_1\times R_2}( N\times 0,I)=0$. Then, by
\cite[Proposition 3.6]{Bennis and Mahdou1}, $N\times0$ is an
$SG$-flat $R_1\times R_2$-module. So, from \cite[Theorem
3.7]{Holm}, $M\times0$ is a $G$-flat $R_1\times R_2$-module (since
it is a direct summand of $N\times0$ as $R_1\times R_2$-modules
and since $R_1\times R_2$ is coherent). Now, we claim that $M$ is
an $SG$-flat $R_1$-module. The $R_1\times R_2$-module $M\times 0$
is $SG$-flat (by the hypothesis condition and since we have proved
that $R_1\times R_2$-module $M\times 0$ is $G$-flat). Then,  there
exists
    a short exact sequence of $R_1\times R_2$-modules $(\star)\quad 0
    \rightarrow M\times 0\rightarrow F \rightarrow
    M \times 0 \rightarrow 0$ where $F$ is a flat $R_1\times
    R_2$-module. Now, we tensor $(\star)$ by $-\otimes_{R_1\times R_2} R_1$ (see that $R_1$ is a projective $R_1\times R_2$-module), we obtain the short exact sequence of $R_1$-modules:
 $$(\star\star)\quad 0
    \rightarrow M\times 0 \displaystyle\otimes _{R_1\times R_2}R_1 \rightarrow F\displaystyle\otimes _{R_1\times R_2}R_1 \rightarrow
    M\times 0 \displaystyle\otimes _{R_1\times R_2}R_1 \rightarrow
    0$$ But  $M\times 0 \displaystyle\otimes _{R_1\times
    R_2}R_1\cong M\times0\displaystyle\otimes_{R_1\times
    R_2}(R_1\times R_2)/(0\times R_2) \cong M\times 0\displaystyle\cong_RM$ (isomorphism of $R$-modules). Then, we can
    write $(\star\star)$ as $ 0 \rightarrow M
    \rightarrow F\displaystyle\otimes _{R_1\times
    R_2}R_1 \rightarrow M \rightarrow 0$. It is clear that $F\displaystyle\otimes
    _{R_1\times
    R_2}R_1$ is a flat $R_1$-module. So, $M$ is an
    $SG$-flat $R_1$-module (since $G.wdim(R_1)$ is finite and by the same argument as above), as
    desired.\\
   Similarly, we can prove that every $G$-flat $R_2$-module is $SG$-flat.\\

 Conversely, assume that every $G$-flat $R_i$-module is $SG$-flat $R_i$-module
for, $i=1,2$. Let
    $M$ be a $G$-flat $R_1\times R_2$-module. We claim that $M$ is strongly Gorenstein flat. We have the isomorphism of $R_1\times R_2$-modules:
$$M\cong M\otimes_{R_1\times
    R_2}R_1\times R_2\cong M\otimes_{R_1\times
    R_2}(R_1\times0\oplus 0\times R_2)\cong M_1\times M_2$$ where
    $M_i=M\otimes_{R_1\times
    R_2} R_i$ for $i=1,2$ (see \cite[Page 102]{Berrick}). By \cite[Proposition 3.10]{Holm}, for each $i=1,2$,  $M_i$
    is a
    $G$-flat $R_i$-module. Then, $M_i$ is an $SG$-flat $R_i$-module  (by the hypothesis condition).\\
    Let $I$ be an injective $R_1$-module and set $n=G.wdim(R_1)$. Then, using \cite[Theorem 3.14]{Holm}, for
    every $R_1$-module $K$ we have $Tor_{R_1}^{n+1}(K,I)=0$ since $Gfd_R(K)\leq n$. Therefore, $fd_{R_1}(I)\leq
    n$. Similarly, we can prove that every injective $R_2$-module has a finite flat
    dimension. Thus, the family $\{R_i\}_{i=1,2}$ of rings satisfies the conditions of
    \cite[Lemma 3.6]{Bennis and Mahdou3}. Hence,   $M=M_1\times M_2$ is
   an  $SG$-flat  $R_1\times R_2$-module, as desired.
   \end{proof}

   Now we are able to construct a family of non-Noetherian coherent  rings $\{R_i\}_{i>0}$ over which every $G$-flat module is strongly
   Gorenstein flat and such that $i=G.wdim(R_i)<G.gldim(R_i)$ and $wdim(R_i)=\infty$.
   \begin{example}
   Consider a non-semisimple quasi-Frobenius ring $R$ and a semihereditary
ring $S$ which is not hereditary. Then, for every positive integer
$n$, we have:
\begin{enumerate}
    \item $\mathcal{SGF}(R\times S[X_1,X_2, ...,X_n])=\mathcal{GF}(R\times S[X_1,X_2, ...,X_n])$, and
    \item $n+1=G.wdim(R\times S[X_1,X_2, ...,X_n])<G.gldim(R\times S[X_1,X_2, ...,X_n])$ and $wdim(R\times S[X_1,X_2, ...,X_n])=\infty$.
\end{enumerate}
\end{example}
\begin{proof} Since $S$ is semihereditary, the ring  $T:= S[X_1,X_2, ...,X_n]$ is coherent
 and from \cite[Example 3.8]{Bennis and Mahdou3} only the
first assertion  need an argument. Since $wdim(T)<\infty$ (by the
Hilbert Syzygies's Theorem), every Gorenstein flat $T$-module is
flat (by \cite[Corollary 3.8]{Holm}).  Hence,
$\mathcal{F}(T)=\mathcal{SGF}(T)=\mathcal{GF}(T)$.
  On the other hand, from Remark \ref{remark example}, $\mathcal{SGF}(R)=\mathcal{GF}(R)$. Thus, since $G.wdim(R)=G.gldim(R)=0$ and $R$ is Noetherian (since $R$ is quasi-Frobenius), Theorem \ref{direct product of
 rings2} finish the proof.
\end{proof}




\begin{thebibliography}{999}\addcontentsline{toc}{section}{\protect\numberline{}{Bibliography}}

\bibitem{Bass} H. Bass; \textit{finitistic dimension and a homological generalization
of semi-primary rings}, Trans. Amer. Math. Soc. 95 (1960),
466-488.
\bibitem{Bennis and Mahdou1} D. Bennis and N. Mahdou; \textit{Strongly Gorenstein projective,
injective, and flat modules}, J. Pure Appl. Algebra 210 (2007),
437-445.

\bibitem{Bennis and Mahdou2} D. Bennis and N. Mahdou; \textit{Global
Gorenstein Dimensions}, Accepted for for publication in Proc.
Amer. Math. Soc. Available from math.AC/0611358 v4 30 Jun 2009.


\bibitem{Bennis and Mahdou3}D. Bennis and N.
Mahdou; \textit{Global Gorenstein dimensions of polynomial rings
and of direct products of rings}, Accepted for publication in
Houston Journal of Mathematics. Available from math.AC/0712.0126
v1 2 Dec 2007.

\bibitem{Ouarghi}D. Bennis, N. Mahdou and K.
Ouarghi; \textit{ Rings over which all modules are strongly
Gorenstein projective}, Accepted for publication in Rocky Mountain
Journal of Mathematics. Available from math.AC/0712.0127 v1 2 Dec
2007.
\bibitem{Berrick} A.J. Berrick and M.E. Keating; \textit{ An introduction to rings and
 modules}, Cambridge university Press, Cambridge, 2000.
\bibitem{Chen} J. Chen and N. Ding; \textit{Coherent rings with finite self-FP-injective dimension}, Comm. Algebra
24 (9) (1996), 2963-2980.
\bibitem{Christensen} L. W. Christensen; \textit{Gorenstein dimensions}, Lecture Notes in Math., Vol. 1747, Springer,
Berlin, (2000).
\bibitem{Christensen Frankild} L. W. Christensen, A. Frankild, and H.
Holm; \textit{On Gorenstein projective, injective and flat
dimensions - a functorial description with applications}, J.
Algebra 302 (2006), 231-279.

\bibitem{Ding} N. Q. Ding and J. L. Chen; \textit{The flat dimensions of injective modules}, Manuscripta Math. 78 (1993), 165-177.
\bibitem{Enochs} E. Enochs and O. Jenda; \textit{On Gorenstein injective modules}, Comm.
Algebra 21 (1993), no. 10, 3489-3501.

\bibitem{Enochs2} E. Enochs and O. Jenda; \textit{Gorenstein injective and projective
modules}, Math. Z. 220 (1995), no. 4, 611-633.
\bibitem{Eno Jenda Torrecillas} E. Enochs, O.
Jenda and B. Torrecillas; \textit{Gorenstein flat modules},
Nanjing Daxue Xuebao Shuxue Bannian Kan 10 (1993), no. 1, 1-9.


\bibitem{Glaz} S. Glaz; \textit{Commutative Coherent Rings}, Springer-Verlag, Lecture Notes
in Mathematics, 1371 (1989).

\bibitem{Holm} H. Holm; \textit{Gorenstein homological dimensions}, J. Pure Appl. Algebra 189 (2004),
167-193.









\end{thebibliography}
\end{document}